\documentclass[10pt]{amsart}
\usepackage{amsmath,amssymb,mathrsfs,color}

\let\cal=\mathcal

\def\R{{\mathbb R}}
\def\H{{\mathbb H}}
\def\P{{\mathbf P}}
\def\L{{\mathcal L}}

\newtheorem{thm}{Theorem}[section]

\newtheorem{lem}[thm]{Lemma}
\newtheorem{prop}[thm]{Proposition}

\theoremstyle{definition}
\newtheorem{de}[thm]{Definition}
\theoremstyle{remark}
\newtheorem{exam}[thm]{Example}
\newtheorem{rem}[thm]{Remark}
\numberwithin{equation}{section}

\newcommand{\rmd}{{\rm d}}
\newcommand{\rme}{{\rm e}}
\newcommand{\ud}{{\rm d}}
\allowdisplaybreaks

\begin{document}

\title[Almost periodic solutions for L\'evy SDE with exponential dichotomy]{Almost periodic solutions for stochastic differential equations with
exponential dichotomy \\
driven by L\'evy noise}


\author{Yan Wang}
\address{Y. Wang: JLU-GT Joint Institute for Theoretical Sciences,
Jilin University, Changchun 130012, P. R. China}
\email{wyan10@mails.jlu.edu.cn}

\thanks{This work is part of the author's Ph.D thesis which is completed on April 2013.
This work is partially supported by NSFC Grants 11271151, 11101183. }

\date{}



\begin{abstract}
In this paper, we study almost periodic solutions for semilinear
stochastic differential equations driven by L\'{e}vy noise with
exponential dichotomy property. Under suitable conditions on the
coefficients, we obtain the existence and uniqueness of bounded
solutions. Furthermore, this unique bounded solution is almost
periodic in distribution under slightly stronger conditions. We also
give two examples to illustrate our results.
\end{abstract}

\maketitle

\section{Introduction}
\setcounter{equation}{0}

Since it was introduced in the 1920s by Bohr and Bochner, almost
periodicity has been extensively studied in differential equations
and dynamical systems. Following the pioneering work of Favard
\cite{Fav}, the separation condition was adopted by Amerio
\cite{Am}, Fink \cite{Fin72}, Seifert \cite{Sei} et al to study the
almost periodic solutions to differential equations. Apart from the
separation method, there are also other methods to investigate
almost periodic solutions, e.g. the stability or Lyapunov function
method in dynamical systems by Miller \cite{M}, Yoshizawa \cite{Y}
et al, the skew-product flow method by Sacker and Sell \cite{SS},
Shen and Yi \cite{SY} et al, the fixed point method by Coppel
\cite{Cop} et al.

From 1980s, the almost periodic solutions to stochastic differential
equations driven by Gaussian noise have been well developed, see
Arnold and Tudor \cite{AT}, Bezandry and Diagana \cite{BD,BD11}, Da
Prato and Tudor \cite{DT} and Halanay \cite{Hal},  Tudor \cite{T},
among others.

L\'evy processes are stochastic processes with independent and
stationary increments, and have many applications ranging from
physics, biology, to finance etc. There are many well known examples
of L\'evy processes, such as Wiener processes, Poisson processes and
stable processes. We refer the reader to Sato \cite{S} for the
theory of L\'evy processes, to Applebaum \cite{A} and Peszat and
Zabczyk \cite{PZ} for finite and infinite dimensional stochastic
differential equations with L\'evy noise perturbations,
respectively. Wang and Liu  \cite{WL} investigated almost periodic
solutions in square-mean sense to stochastic differential equations
perturbed by L\'evy noise. Indeed, as indicated in \cite{KMR,MR}, it
appears that almost periodicity in distribution sense is a more
appropriate concept relatively to solutions of stochastic
differential equations. Recently, Liu and Sun \cite{LS} and Sun
\cite{Sun} studied almost automorphic in distribution solutions of
stochastic differential equations with L\'evy noise.

In this paper, we study the existence and uniqueness of bounded
solution that is almost periodic in distribution to semilinear
stochastic differential equations with exponential dichotomy
property driven by infinite dimensional L\'evy processes which may
admit large jumps.

The rest of the paper is organized as follows. In Section 2, we
recall some basic definitions and  results of L\'evy processes and
almost periodic processes. In Sections 3, we investigate the
existence and uniqueness of bounded solution for semilinear
stochastic differential equations  with exponential dichotomy
property driven by L\'{e}vy noise. Moreover, the unique bounded
solution is almost periodic in distribution under suitable
conditions. In Section 4, we give two examples to illustrate the
theoretical results obtained in this paper.

\section{Preliminaries}
\setcounter{equation}{0}

Throughout the rest of this paper, we assume that $(\mathbb{H}, \|\cdot\|)$ and
$(U,|\cdot|_U)$ are real separable Hilbert spaces. We denote by
$L(U,\mathbb{H})$ the space of all bounded linear operators from $U$
to $\mathbb{H}$. Note that $L(U,\mathbb{H})$ is a Banach space, and
we denote the norm by $\|\cdot\|_{L(U,\mathbb{H})}$. Let
$(\Omega, \mathcal{F}, \mathbf{P} )$ be a probability space, and
$\mathcal{L}^{2}(\mathbf{P}, \mathbb{H})$ stand for the space of
all $\mathbb{H}$-valued random variables $Y$ such that
\begin{displaymath}
\mathbf{E} \|Y\|^{2} = \int_{\Omega}\|Y\|^{2}\ud \mathbf{ P}<\infty.
\end{displaymath}
For $Y\in\mathcal{L}^{2}(\mathbf{P},\mathbb{H})$, let
\[
\|Y\|_{2}{:=}\left( \int_{\Omega}\|Y\|^{2}\ud \mathbf{P} \right)^{
1/2 }.
\]
Then $\mathcal{L}^{2} (\mathbf{P}, \mathbb{H})$ is a Hilbert space
equipped with the norm $\|\cdot\|_2$.
The L\'evy process we consider is $U$-valued.

\subsection{L\'evy process and L\'evy-It\^o decomposition}\label{levy}

\begin{de}\rm
A $U$-valued stochastic process $L=(L(t),t\ge 0)$ is called
\emph{L\'{e}vy process} if:
\begin{enumerate}
  \item $L(0)=0$ almost surely;
  \item $L$ has independent and stationary increments;
  \item $L$ is stochastically continuous,
  i.e. for all $\epsilon>0$ and for all $s>0$
   \[
     \lim_{t\rightarrow s}P(|L(t)-L(s)|_U>\epsilon)=0.
  \]
\end{enumerate}
\end{de}
Let $L$ be an L\'evy process. We recall some definitions and L\'evy-It\^o decomposition theorem; see \cite{A,PZ,S}
for details.
\begin{de}
\begin{enumerate}
  \item[ ]
  \item A Borel set $B$ in $U-\{0\}$ is {\em bounded below}  if $0\notin \overline{B}$,
         the closure of $B$.
  \item $\nu(\cdot)=\mathbf{E}(N(1,\cdot))$ is called the
        \emph{intensity measure} associated with $L$,
        where $\Delta L(t)=L(t)-L(t-)$ for each $t\geq 0$ and
        \[
           N(t,B)(\omega):=\sharp \{0\leq s\leq t: \Delta L(s)(\omega)\in B\}
           = \sum_{0\leq s\leq t} \chi_B (\Delta L(s)(\omega))
        \]
        with $\chi_B$ being the indicator function for any Borel set $B$
        in $U-\{0\}$.
  \item $N(t,B)$ is called {\em Poisson random measure} if $B$ is bounded below, for each $t\geq 0$.
  \item For each $t\geq 0$ and $B$ bounded below, we define
        the \emph{compensated Poisson random measure} by
        \[
        \widetilde{N}(t,B)=N(t,B)-t\nu(B).
        \]
\end{enumerate}
\end{de}

\begin{prop}[L\'evy-It\^o decomposition]\label{lid}
If $L$ is a $U$-valued L\'{e}vy process, then there exist $a\in U$,
a $U$-valued Wiener process $W$ with covariance operator $Q$, so
called $Q$-Wiener process, and an independent Poisson random measure
$N$ on $\R^+ \times (U- \{0\})$ such that, for each $t\ge 0$
\begin{equation}\label{levy-ito}
L(t)=at+W(t)+ \int_{|x|_U <1} x \widetilde N (t,\rmd x)+ \int_{|x|_U
\ge 1} x N (t,\rmd x).
\end{equation}
Here the Poisson random measure $N$ has the intensity measure $\nu$
which satisfies
\begin{equation}\label{nu}
\int_U(|y|_U^2\wedge1)\nu(\rmd y)<\infty
\end{equation}
and $\widetilde N$ is the compensated Poisson random measure of $N$.
\end{prop}

In this paper, Wiener processes we consider are $Q$-Wiener processes;
see \cite{DZ} for details. For simplicity, we
assume that the covariance operator $Q$ of $W$
is of trace class, i.e. ${\rm Tr}Q < \infty$.
Assume that $L_1$ and $L_2$ are two
independent, identically distributed L\'evy processes with
decompositions as in Proposition \ref{lid} with $a, Q, W, N$. Let
\[
L (t) =\left\{ \begin{array}{ll}
                 L_1(t), &  \hbox{ for } t\ge 0, \\
                 -L_2(-t), &  \hbox{ for } t\le 0.
               \end{array}
 \right.
\]
Then $L$ is a two-sided L\'evy process. We assume that the two sided
L\'evy process $L$ is defined on the filtered probability space
$(\Omega,\mathcal F,\mathbf P,(\mathcal F_t)_{t\in\mathbb R})$. By
the L\'evy-It\^o decomposition, it follows that $\int_{|x|_U\ge 1}
\nu(\rmd x)< \infty$. Then we denote $b:=\int_{|x|_U\ge 1} \nu(\rmd
x)$ throughout this paper. We note that the process
$\widetilde{L}:=(\widetilde{L}(t)=L(t+s)-L(s))$ for some $s\in\R$ is
also a two-sided L\'evy process with the same law as $L$.


\subsection{Square-mean almost periodic processes}

\begin{de}\cite{LS}
A stochastic process $Y: \mathbb{R} \rightarrow
\mathcal{L}^{2}(\mathbf{P}, \mathbb{H})$ is said to be {\em
$\mathcal L^2$-continuous} if for any $s\in\R$,
\begin{eqnarray*}
\lim_{t \rightarrow s}\mathbf{E} \|Y(t)-Y(s)\|^{2}
 = 0.
\end{eqnarray*}
It is {\em $\cal L^2$-bounded} if $\sup_{t\in\R} \|Y(t)\|_2 <\infty$.
\end{de}

Note that if an $\mathbb H$-valued process is $\mathcal
L^2$-continuous, then it is necessarily stochastically continuous.

\begin{de}\cite{WL}\label{defap}
(1) An $\mathcal L^2$-continuous stochastic process $x: \mathbb{R}
\rightarrow \mathcal{L}^{2}(\mathbf{P}, \mathbb{H})$ is said to be
{\em square-mean almost periodic} if for every sequence of real
numbers $\{s'_n\}$, there exist a subsequence $\{s_n\}$ and an
$\mathcal L^2$-continuous stochastic process $y: \mathbb{R}
\rightarrow \mathcal{L}^{2}(\mathbf{P}, \mathbb{H})$ such that
\begin{equation*}
\lim_{n \rightarrow \infty}\sup_{t\in\mathbb R} \mathbf{E}
\|x(t+s_{n})-y(t)\|^{2}= 0.
\end{equation*}
The collection of all square-mean almost periodic stochastic
processes $x: \mathbb{R} \rightarrow
\mathcal{L}^{2}(\mathbf{P},\mathbb{H})$ is denoted by
$AP(\mathbb{R}; \mathcal{L}^{2}(\mathbf{P}, \mathbb{H}))$.\\
(2) A continuous function $g: \mathbb{R} \times
\mathcal{L}^{2}(\mathbf{P}, \mathbb{H})\to L(U,\mathcal L^2(\mathbf
P,\mathbb H))$, $(t,Y)\mapsto g(t,Y)$ is said to be {\em uniformly
square-mean almost periodic} if for every sequence of real numbers
$\{s'_n\}$, there exist a subsequence $\{s_n\}$ and a continuous
function $\widetilde g: \mathbb{R} \times
\mathcal{L}^{2}(\mathbf{P}, \mathbb{H})\to
L(U,\mathcal{L}^{2}(\mathbf{P}, \mathbb{H}))$ such that
\begin{equation*}
\lim_{n \rightarrow \infty} \sup_{t\in\R,Y\in\cal K}\mathbf{E}
\|g(t+s_{n},Y)-\widetilde g(t,Y)\|_{L(U,\mathcal{L}^{2}(\mathbf{P},
\mathbb{H}))}^{2}= 0
\end{equation*}
for every bounded or compact set $\cal K\subset
\mathcal{L}^{2}(\mathbf{P}, \mathbb{H})$.
\\
(3) A function $F: \mathbb{R} \times \mathcal{L}^{2}(\mathbf{P},
\mathbb{H}) \times U \rightarrow \mathcal{L}^2(\mathbf{P},
\mathbb{H})$, $(t,Y,x)\mapsto F(t,Y,x)$ is said to be {\em uniformly
Poisson square-mean almost periodic} if $F$ is continuous in the
following sense
\[
\int_U \mathbf E \|F(t,Y,x)-F(t',Y',x)\|^2\nu(\rmd x) \to 0 \qquad
\hbox{as } (t',Y')\to (t,Y)
\]
and that for every sequence of real numbers $\{s'_n\}$ and every
bounded or compact set $\cal K\subset \mathcal{L}^{2}(\mathbf{P},
\mathbb{H})$, there exist a subsequence $\{s_n\}$ and a continuous
function $\widetilde{F}: \mathbb{R} \times
\mathcal{L}^{2}(\mathbf{P}, \mathbb{H}) \times U \rightarrow
\mathcal{L}^2(\mathbf{P},\mathbb{H})$ in the above sense, with
$\int_U\mathbf{E}\|\widetilde F(t,Y,x)\|^2\nu(\rmd x)<\infty$, such
that
\begin{equation*}
\lim_{n\rightarrow \infty} \sup_{t\in\R,Y\in\cal K} \int_U\mathbf{E}
\|F(t+s_n,Y,x) -\widetilde{F}(t,Y,x)\|^2\nu(\rmd x) =0
\end{equation*}
for every bounded or compact set $\cal K\subset
\mathcal{L}^{2}(\mathbf{P}, \mathbb{H})$.

\end{de}

In the sequel, (uniformly Poisson) square-mean almost periodicity is
also called (uniformly Poisson) $\cal L^2$-almost periodicity.



\begin{lem}\cite{BD}\label{lem2}
$AP(\mathbb{R};\mathcal{L}^2(\mathbf{P},\mathbb{H}))$ is a Banach
space when it is equipped with the norm
\begin{eqnarray*}
\|Y\|_{\infty}:=\sup_{t\in\mathbb R}\|Y(t)\|_2=\sup_{t\in
\mathbb{R}}(E\|Y(t)\|^2)^{\frac{1}{2}},
\end{eqnarray*}
for $Y\in AP(\mathbb{R};\mathcal{L}^2(\mathbf{P},\mathbb{H}))$.
\end{lem}


\begin{prop}\cite{BD}\label{aacom}
Let $f: \mathbb{R}\times \mathcal{L}^2(\mathbf{P}, \mathbb{H})
\rightarrow \mathcal{L}^2(\mathbf{P}, \mathbb{H})$, $(t,Y)\mapsto
f(t,Y)$ be square-mean almost periodic on each compact set of
$\mathcal{L}^2(\mathbf{P}, \mathbb{H})$, and assume that $f$
satisfies Lipschitz condition in the following sense:
\begin{eqnarray*}
\mathbf{E}\|f(t,Y)-f(t,Z)\|^2\leq L\mathbf{E}\|Y-Z\|^2
\end{eqnarray*}
for all $Y,Z \in \mathcal{L}^2(\mathbf{P}, \mathbb{H})$ and $t\in
\mathbb{R}$, where $L>0$ is independent of $t$. Then for any
square-mean almost periodic process $Y:\mathbb{R}\rightarrow
\mathcal{L}^2(\mathbf{P}, \mathbb{H})$,  the stochastic process
$F:\mathbb{R}\rightarrow \mathcal{L}^2(\mathbf{P}, \mathbb{H})$
given by $F(t):=f(t,Y(t))$ is square-mean almost periodic.
\end{prop}


\begin{prop}\cite{WL}\label{lem}
If $F, F_{1}, F_{2}: \mathbb{R} \times \mathcal{L}^{2}(\mathbf{P},
\mathbb{H}) \times U \rightarrow \mathcal{L}^2(\mathbf{P},
\mathbb{H})$ are uniformly Poisson square-mean almost periodic
functions on each compact set of $\mathcal{L}^{2}(\mathbf{P},
\mathbb{H})$, then
\begin{enumerate}
  \item $F_{1}+F_{2}$ is Poisson square-mean almost periodic.
  \item $\lambda F$ is Poisson square-mean almost periodic for every scalar
         $\lambda$.
  \item For any compact subset $\cal K\subset \mathcal{L}^{2}(\mathbf{P}, \mathbb{H})$, there exists a constant $M=M(\cal K)>0$ such that
\[
\sup_{t\in \mathbb{R},Y\in\cal
K}\int_U\mathbf{E}\|F(t,Y,x)\|^2\nu(\rmd x) \leq M.
\]
\end{enumerate}
\end{prop}

\begin{prop}\cite{WL}\label{th}
Let $F: \mathbb{R} \times \mathcal{L}^{2}(\mathbf{P}, \mathbb{H})
\times U \rightarrow \mathcal{L}^2(\mathbf{P}, \mathbb{H})$,
$(t,Y,x)\mapsto F(t,Y,x)$ be uniformly Poisson square-mean almost
periodic on any compact subsets of $\mathcal{L}^2(\mathbf{P},
\mathbb{H})$, and assume that $F$ satisfies the Lipschitz condition
in the follow sense:
\begin{eqnarray*}
\int_U\mathbf{E}\|F(t,Y,x)-{F}(t,Z,x)\|^2\nu(\rmd x)\leq L\mathbf{E}\|Y-Z\|^2
\end{eqnarray*}
for all $Y,Z\in \mathcal{L}^2(\mathbf{P}, \mathbb{H})$ and $t\in
\mathbb{R}$, where $L>0$ is independent of $t$. Then for any
square-mean almost periodic process $Y:\mathbb{R}\rightarrow
\mathcal{L}^2(\mathbf{P}, \mathbb{H})$, the function $D:\mathbb{R}
\times U\rightarrow \mathcal{L}^2(\mathbf{P}, \mathbb{H})$ given by
$D(t,x):= F(t,Y(t),x)$ is Poisson square-mean almost periodic.
\end{prop}

\subsection{Almost periodicity in distribution}
Let $\cal P(\mathbb H)$ be the space of all Borel probability measures on $\mathbb H$ endowed
with the $\beta$ metric:
\[
\beta (\mu,\nu) :=\sup\left\{ \left| \int f \rmd \mu - \int f\rmd \nu\right|: \|f\|_{BL} \le 1
\right\}, \quad \mu,\nu\in \cal P(\mathbb H),
\]
where $f$ are Lipschitz continuous real-valued functions on $\mathbb H$ with the norms
\[
\|f\|_{BL}= \|f\|_L + \|f\|_\infty, \|f\|_L=\sup_{x\neq y} \frac{|f(x)-f(y)|}{\|x-y\|}, \|f\|_{\infty}=\sup_{x\in\mathbb H}|f(x)|.
\]
See \cite[\S 11.3]{Dud} for $\beta$ metric and related properties.
Recall that a sequence $\{\mu_n\}\subset \cal P(\mathbb H)$ is said
to weakly converge to $\mu$ if $\int f \rmd\mu_n\to \int f\rmd \mu$
for all $f\in C_b(\mathbb H)$, the space of all bounded continuous
real-valued functions on $\mathbb H$. It is well known that the
$\beta$ metric is a complete metric on $\cal P(\mathbb H)$ and that
a sequence $\{\mu_n\}$ weakly converges to $\mu$ if and only if
$\beta(\mu_n,\mu)\to 0$ as $n\to\infty$.

\begin{de}
An $\mathbb H$-valued stochastic process $Y(t)$ is said to be {\em
almost periodic in distribution} if its law $\mu(t)$ is a $\cal
P(\mathbb H)$-valued almost periodic mapping, i.e. for every
sequence of real numbers $\{s'_n\}$, there exist a subsequence
$\{s_n\}$ and a $\cal P(\mathbb H)$-valued continuous mapping
$\tilde \mu(t)$ such that
\begin{equation*}
\lim_{n \rightarrow \infty}\sup_{t\in\mathbb R}
\beta(\mu(t+s_n),\tilde \mu(t))= 0.
\end{equation*}
\end{de}

\begin{rem}\label{aad}
Recall that a sequence of random variables $\{X_n\}$ is said to
converge in distribution to the random variable $X$ if the
corresponding laws $\{\mu_n\}$ weakly converge to the law $\mu$ of
$X$, i.e. $\beta(\mu_n,\mu)\to 0$. Since $\cal L^2$ convergence
implies convergence in distribution for a sequence of random
variables, a square-mean almost periodic stochastic process is
necessarily an almost periodic in distribution one; but the converse
is not true.
\end{rem}

\section{Main Results}
\setcounter{equation}{0}

Consider the following semilinear stochastic differential equation
\begin{eqnarray}\label{dxr}
\rmd Y(t)&=&AY(t)\rmd t+f(t,Y(t))\rmd t+g(t,Y(t))\rmd W(t)\\
&&+\int_{|x|_U<1}F(t,,Y(t),x)\widetilde{N}(\ud t,\ud x)\nonumber\\
&&+\int_{|x|_U\geq 1}G(t,Y(t),x){N}(\ud t,\ud x), \qquad t \in
\mathbb{R}, \nonumber
\end{eqnarray}
where $A$ is an infinitesimal generator which generates a
$C^0$-semigroup $\{T(t)\}_{t\ge 0}$ on $\mathbb H$, $f:
\mathbb{R}\times\mathcal{L}^{2}(\mathbf{P},\mathbb{H})
\rightarrow\mathcal{L}^{2}(\mathbf{P},\mathbb{H})$,
$g:\mathbb{R}\times \mathcal{L}^{2}(\mathbf{P},\mathbb{H})
\rightarrow L(U,\mathcal{L}^{2}(\mathbf{P},\mathbb{H}))$,
$F:\mathbb{R}\times\mathcal{L}^{2}(\mathbf{P},\mathbb{H})\times U
\rightarrow\mathcal{L}^{2}(\mathbf{P},\mathbb{H})$,
$G:\mathbb{R}\times\mathcal{L}^{2}(\mathbf{P},\mathbb{H})\times U
\rightarrow\mathcal{L}^{2}(\mathbf{P},\mathbb{H})$ are continuous
functions; $W$ and $N$ are the L\'evy-It\^o decomposition components
of the two-sided L\'evy process.

\begin{de} The semigroup $\{T(t)\}_{t\geq0}$ is said to satisfy the {\em exponential dichotomy property} if
there exist projection $P$ and constants $K,\omega>0$ such that
$T(t)$ commutes with $P$ for each $t\ge 0$, $\ker(P)$ is invariant
with respect to $T(t)$, $T(t):R(J)\rightarrow R(J)$ is invertible
and the following holds
\begin{equation}\label{ed}
\|T(t)Px\|\leq K\exp(-\omega t)\|x\|,\;\;\;\textrm{ $t\ge0$};\;\;\;
\|T(t)Jx\|\leq K\exp(\omega t)\|x\|,\;\;\;\textrm{ $t\leq0$},
\end{equation}
where $J=I-P$ and  $T(t)=(T(-t))^{-1}$ for $t\leq0$.
\end{de}

\begin{de}\label{def}
An $\mathcal F_t$-progressively measurable stochastic process
$\{Y(t)\}_{t\in \mathbb{R}}$ is called a {\em mild solution} of
\eqref{dxr} if it satisfies the corresponding stochastic integral
equation
\begin{align}\label{solutionr}
Y(t)=&T(t-r)Y(r)+\int_r^tT(t-s)f(s,Y(s-))\rmd s
+\int_r^tT(t-s)g(s,Y(s-))\rmd W(s)\nonumber\\
&+\int_{r}^{t}\int_{|x|_U<1}T(t-s)F(s,Y(s-),x)\widetilde{N}(\rmd s,\rmd x)\nonumber\\
&+\int_{r}^{t}\int_{|x|_U\geq 1}T(t-s)G(s,Y(s-),x){N}(\rmd s,\rmd
x),
\end{align}
for all $t\ge r$ and each $r\in\mathbb{R}$.
\end{de}

In this section, we require the following basic assumptions.

\begin{enumerate}
  \item[(H1)] The semigroup $T(t)$
              generated by ${A}$ satisfies the exponential dichotomy property, that is,
              \eqref{ed} holds.

  \item[(H2)] Assume that $f$, $g$ are uniformly square-mean almost periodic and $F$, $G$ are uniformly Poisson square-mean
              almost periodic.

  \item[(H3)] Assume that $f$, $g$, $F$ and
              $G$ satisfy Lipschitz conditions in $Y$ uniformly with respect to $t$, that is,
              for all $Y, Z\in \mathcal{L}^2(\mathbf{P},\mathbb{H})$ and $t\in
              \mathbb{R}$,
              \begin{equation*}
               \mathbf{E}\|f(t,Y)-f(t,Z)\|^2\leq L \mathbf{E}\|Y-Z\|^2,
              \end{equation*}
              \begin{equation*}
               \mathbf{E}\|(g(t,Y)-g(t,Z))Q^{1/2}\|_{L(U,\mathcal{L}^{2}(\mathbf{P},
               \mathbb{H}))}^2\leq L \mathbf{E}\|Y-Z\|^2,
              \end{equation*}
              \begin{equation}\label{F}
                \int_{|x|_U<1}\mathbf{E}\|F(t,Y,x)-F(t,Z,x)\|^2\nu(\ud x)\leq L
                \mathbf{E}\|Y-Z\|^2,
              \end{equation}
              \begin{equation}\label{G}
                 \int_{|x|_U\geq1}\mathbf{E}\|G(t,Y,x)-G(t,Z,x)\|^2\nu(\ud x)\leq L
                 \mathbf{E}\|Y-Z\|^2,
              \end{equation}
              for some constant $L>0$ independent of $t$.

\end{enumerate}

\begin{lem}\label{leml2}
Assume (H1)-(H3). If \eqref{dxr} has an
 $\cal L^2$-bounded mild solution, then this solution is $\cal L^2$-continuous.
\end{lem}

\begin{proof}
Since $T(t)$ is a $C^0$-semigroup, it follows from \cite[Chapter 1,
Theorem 2.2]{Pazy} that there exist positive constants $M,\delta$
such that $\|T(t)\|\le M \rme^{\delta t}$ for all $t\ge 0$. If
$Y(t)$ is an $\cal L^2$-bounded solution of \eqref{dxr}, i.e.
\eqref{solutionr} holds, then it follows from the \eqref{ed},
It\^o's isometry, the properties of the integral for the Poisson
random measure and Cauchy-Schwarz inequality, we have for $t\ge r$
\begin{align*}
&\mathbf E \|Y(t)-Y(r)\|^2 \\
\le & 5 \mathbf{E}\|T(t-r)Y(r)-Y(r)\|^2\\
&+5\mathbf{E}\left\|\int_r^t T(t-s)f(s,Y(s-))\rmd s\right\|^2 \\
&+5 \mathbf{E}\left\|\int_r^tT(t-s)g(s,Y(s-))\rmd W(s)\right\|^2 \\
&+5\mathbf{E} \left\| \int_{r}^{t}\int_{|x|_U<1}T(t-s)F(s,Y(s-),x)\widetilde{N}(\rmd s,\rmd x)\right\|^{2}\\
&+5\mathbf{E}\left\| \int_{0}^{t}\int_{|x|_U\geq
1}T(t-s)G(s,Y(s-),x){N}(\rmd s,\rmd x)\right\|^{2}\\
\le & 5 \|T(t-r)-Id\|^2\cdot \mathbf{E}\|Y(r)\|^2\\
&+5 M^2 \rme^{2\delta(t-r)} \int_r^t 1 \rmd s \cdot \int_r^t \mathbf{E}\|f(s,Y(s-))\|^2\rmd s\\
&+5 M^2 \rme^{2\delta(t-r)} \int_r^t  \mathbf{E} \|g(s,Y(s-))\|^2\rmd s \\
&+5 M^2 \rme^{2\delta(t-r)}  \int_{r}^{t} \int_{|x|_U<1} \mathbf{E} \|F(s,Y(s-),x)\|^2\nu(\rmd x)\rmd s\\
&+10 M^2 \rme^{2\delta(t-r)} \int_{r}^{t}   \int_{|x|_U\geq
1} \mathbf{E} \|G(s,Y(s-),x)\|^2 \nu(\rmd x)\rmd s\\
&+ 10  M^2 \rme^{2\delta(t-r)} \int_{r}^{t}   \int_{|x|_U\geq 1} 1
\nu(\rmd x)\rmd s \cdot \int_{r}^{t}  \int_{|x|_U\geq 1} \mathbf{E}
\|G(s,Y(s-),x)\|^2 \nu(\rmd x)\rmd s.
\end{align*}
Note that $\|T(t-r)-Id\|^2\to 0$ when $t\to r$, it follows from the
Lipschitz property of $f$, the almost periodicity of $f$ in $s$ and
the $\cal L^2$-boundedness of  $Y(\cdot)$ that
\begin{align*}
\sup_{s\in \R}\|f(s,Y(s))\|_2  \le & \sup_{s\in\R}\|f(s,Y(s)) - f(s, 0)\|_2 + \sup_{s\in\R}\|f(s,0)\|_2 \\
\le &  \sqrt{L}\sup_{s\in\R} \|Y(s)\|_2+ \sup_{s\in\R}\|f(s,0)\|_2 <
\infty;
\end{align*}
similarly
\[
\sup_{s\in \R}\|g(s,Y(s))\|_2 < \infty.
\]
By the Lipschitz property and Poisson almost periodicity of $F$, and
the $\cal L^2$-boundedness of  $Y(\cdot)$, we have
\begin{align*}
\sup_{s\in \R} \int_{|x|_U<1} \mathbf{E} \|F(s,Y(s),x) \nu(\rmd
x)\|^2  \le & \sup_{s\in\R} \int_{|x|_U<1} \mathbf{E}\|F(s,Y(s),x) -
F(s, 0,x)\|^2
\nu(\rmd x)\\
& + \sup_{s\in\R} \int_{|x|_U<1} \mathbf{E}\|F(s,0,x)\|^2 \nu(\rmd x) \\
\le &  L \sup_{s\in\R} \mathbf{E}\|Y(s)\|^2+ M (0) < \infty;
\end{align*}
similarly,
\begin{align*}
\sup_{s\in \R} \int_{|x|_U\ge 1} \mathbf{E} \|G(s,Y(s),x) \nu(\rmd
x)\|^2  < \infty.
\end{align*}
Therefore,
\[
\mathbf{E} \|Y(t)-Y(r)\|^2 \to 0 \qquad \hbox{当 } t\to r.
\]
That is, $Y(\cdot)$ is $\cal L^2$-continuous.
\end{proof}

\begin{thm}\label{thnon1}
Let (H1)-(H3) be satisfied.  Then \eqref{dxr} has a unique $\cal
L^2$-bounded mild solution, provided
\begin{equation}\label{Bar}
\frac{1+2b}{\omega^{2}} + \frac{2}{\omega}<\frac{1}{16K^{2}L}.
\end{equation}
\end{thm}

\begin{proof}

Firstly, note that if $Y:\mathbb R\to \mathcal{L}^{2}(\mathbf{P},
\mathbb{H})$ is bounded, i.e. $\|Y\|_{\infty} <\infty$, then $Y(t)$
is a mild solution of \eqref{dxr} if and only if it staisfies the
following integral equation
\begin{align}\label{solur}
Y(t)=&\int_{-\infty}^tT(t-s)Pf(s,Y(s-))\rmd
s+\int_{-\infty}^tT(t-s)Pg(s,Y(s-))\rmd W(s)\nonumber\\
&+\int_{-\infty}^{t}\int_{|x|_U<1}T(t-s)PF(s,Y(s-),x)\widetilde{N}(\rmd s,\rmd x)\nonumber\\
&+\int_{-\infty}^{t}\int_{|x|_U\geq 1}T(t-s)PG(s,Y(s-),x){N}(\rmd
s,\rmd x)\nonumber
\\
&-\int_{t}^{+\infty} T(t-s)Jf(s,Y(s-))\rmd
s-\int_{t}^{+\infty} T(t-s)Jg(s,Y(s-))\rmd W(s)\nonumber\\
&-\int_{t}^{+\infty}\int_{|x|_U<1}T(t-s)JF(s,Y(s-),x)\widetilde{N}(\rmd s,\rmd x)\nonumber\\
&-\int_{t}^{+\infty}\int_{|x|_U\geq 1}T(t-s)JG(s,Y(s-),x){N}(\rmd
s,\rmd x).
\end{align}

Consider the nonlinear operator $\mathcal S$ acting on the Banach
space $C_b(\mathbb R;\mathcal{L}^2(\mathbf{P},\mathbb{H}))$ which is
given by
\begin{eqnarray*}
(\mathcal{S}Y)(t)&:=&\int_{-\infty}^tT(t-s)Pf(s,Y(s-))\rmd
s+\int_{-\infty}^tT(t-s)Pg(s,Y(s-))\rmd W(s)\\
&&+\int_{-\infty}^{t}\int_{|x|_U<1}T(t-s)PF(s,Y(s-),x)\widetilde{N}(\rmd s,\rmd x) \\
&&+\int_{-\infty}^{t}\int_{|x|_U\geq 1}T(t-s)PG(s,Y(s-),x){N}(\rmd
s,\rmd x)
\\
&&-\int_{t}^{+\infty} T(t-s)Jf(s,Y(s-))\rmd
s-\int_{t}^{+\infty} T(t-s)Jg(s,Y(s-))\rmd W(s)\nonumber\\
&&-\int_{t}^{+\infty}\int_{|x|_U<1}T(t-s)JF(s,Y(s-),x)\widetilde{N}(\rmd s,\rmd x)\nonumber\\
&&-\int_{t}^{+\infty}\int_{|x|_U\geq 1}T(t-s)JG(s,Y(s-),x){N}(\rmd
s,\rmd x)\\
& =: & (\mathcal S_1 Y)(t) + (\mathcal S_2 Y)(t)+(\mathcal S_3
Y)(t)+(\mathcal S_4 Y)(t),
\end{eqnarray*}
where
\begin{eqnarray*}
(\mathcal S_1 Y)(t)&:=&\int_{-\infty}^t T(t-s)Pf(s,Y(s-))\rmd s
-\int_{t}^{+\infty} T(t-s)Jf(s,Y(s-))\rmd s,\\
(\mathcal S_2 Y)(t)&:=&\int_{-\infty}^t T(t-s)Pg(s,Y(s-))\rmd W(s)
-\int_{t}^{+\infty} T(t-s)Jg(s,Y(s-))\rmd W(s),\\
(\mathcal S_3 Y)(t)&:=&\int_{-\infty}^t\int_{|x|_U<1}
T(t-s)PF(s,Y(s-),x)\widetilde{N}(\rmd s,\rmd x)\\
&&-\int_{t}^{+\infty}\int_{|x|_U<1}T(t-s)JF(s,Y(s-),x)\widetilde{N}(\rmd s,\rmd x),\\
(\mathcal S_4 Y)(t)&:=&\int_{-\infty}^t\int_{|x|_U\geq 1}T(t-s)PG(s,Y(s-),x)N(\rmd s,\rmd x)\\
&&-\int_{t}^{+\infty} \int_{|x|_U\geq 1}T(t-s)JG(s,Y(s-),x){N}(\rmd
s,\rmd x).
\end{eqnarray*}
Since $T(\cdot)$ satisfies the exponential dichotomy property, $f,g$
are uniformly $\cal L^2$-almost periodic, $F,G$ are uniformly
Poisson $\cal L^2$-almost periodic, and they satisfy the Lipschitz
property, it follows that the operator $\cal S$ maps $C_b(\mathbb
R;\mathcal{L}^2(\mathbf{P},\mathbb{H}))$ to itself. If $\cal S$ is
contraction mapping, then the Banach fixed point theorem yields that
\eqref{dxr} admits a unique $\cal L^2$-bounded and $\cal
L^2$-continuous mild solution.

Now we will show that $\mathcal S$ is a contraction mapping on
$C_b(\mathbb R;\mathcal{L}^2(\mathbf{P},\mathbb{H}))$. For
$Y_1,Y_2\in C_b(\mathbb R;\mathcal{L}^2(\mathbf{P},\mathbb{H}))$ and
$t\in\R$ we have
\begin{eqnarray*}
&&\mathbf{E}\|(\mathcal{S}Y_1)(t)-(\mathcal{S}Y_2)(t)\|^2\\
&=&\mathbf{E}\bigg\|\int_{-\infty}^tT(t-s)P[f(s,Y_1(s-))-f(s,Y_2(s-))]\rmd s\\
&&+\int_{-\infty}^tT(t-s)P[g(s,Y_1(s-))-g(s,Y_2(s-))]\rmd W(s)\\
&&+\int_{-\infty}^{t}\int_{|x|_U<1}T(t-s)P[F(s,Y_{1}(s-),x)-F(s,Y_{2}(s-),x)]\widetilde{N}(\rmd s,\rmd x)\\
&&+\int_{-\infty}^{t}\int_{|x|_U\geq
1}T(t-s)P[G(s,Y_{1}(s-),x)-G(s,Y_{2}(s-),x)]{N}(\rmd s,\rmd x)
        \\
&&-\int_{t}^{+\infty} T(t-s)J[f(s,Y_1(s-))-f(s,Y_2(s-))]\rmd s\\
&&-\int_{t}^{+\infty} T(t-s)J[g(s,Y_1(s-))-g(s,Y_2(s-))]\rmd W(s)\\
&&-\int_{t}^{+\infty}\int_{|x|_U<1}T(t-s)J[F(s,Y_{1}(s-),x)-F(s,Y_{2}(s-),x)]\widetilde{N}(\rmd s,\rmd x)\\
&&-\int_{t}^{+\infty}\int_{|x|_U\geq
1}T(t-s)J[G(s,Y_{1}(s-),x)-G(s,Y_{2}(s-),x)]{N}(\rmd s,\rmd
x)\bigg\|^{2}
\\
&\leq& 8\mathbf{E}\left\|\int_{-\infty}^t T(t-s)P[f(s,Y_1(s-))-f(s,Y_2(s-))]\rmd s\right\|^2\\
&&+8\mathbf{E}\left\|\int_{-\infty}^tT(t-s)P[g(s,Y_1(s-))-g(s,Y_2(s-))]\rmd W(s)\right\|^2\\
&&+8\mathbf{E}\left\|
\int_{-\infty}^{t}\int_{|x|_U<1}T(t-s)P[F(s,Y_{1}(s-),x)-F(s,Y_{2}(s-),x)]
\widetilde{N}(\rmd s,\rmd x)\right\|^{2} \\
&&+8\mathbf{E}\left\| \int_{-\infty}^{t}\int_{|x|_U\geq
1}T(t-s)P[G(s,Y_{1}(s-),x)-G(s,Y_{2}(s-),x)]{N}(\rmd s,\rmd
x)\right\|^{2}
     \\
&&+8\mathbf{E}\left\|\int_{t}^{+\infty} T(t-s)J[f(s,Y_1(s-))-f(s,Y_2(s-))]\rmd s\right\|^2\\
&&+8\mathbf{E}\left\|\int_{t}^{+\infty} T(t-s)J[g(s,Y_1(s-))-g(s,Y_2(s-))]\rmd W(s)\right\|^2\\
&&+8\mathbf{E}\left\|\int_{t}^{+\infty}
\int_{|x|_U<1}T(t-s)J[F(s,Y_{1}(s-),x)-F(s,Y_{2}(s-),x)]\widetilde{N}(\rmd s,\rmd x)\right\|^{2} \\
&&+8\mathbf{E}\left\|\int_{t}^{+\infty} \int_{|x|_U\geq
1}T(t-s)J[G(s,Y_{1}(s-),x)-G(s,Y_{2}(s-),x)]{N}(\rmd s,\rmd
x)\right\|^{2} .
\end{eqnarray*}

Similar to the proof of \cite[Theorem 3.2]{BD}, it follows from the
Cauchy-Schwarz inequality that we have the following estimates for
the 1st and the 5th term on the right hand side of the above
inequality
\begin{eqnarray*}
\mathbf{E}\left\|\int_{-\infty}^t
T(t-s)P[f(s,Y_1(s-))-f(s,Y_2(s-))]\rmd s\right\|^2&\leq&
\frac{K^2L}{\omega^2}\cdot\sup_{s\in
\mathbb{R}}\mathbf{E}\|Y_1(s)-Y_2(s)\|^2,
\\
\mathbf{E}\left\|\int_{t}^{+\infty}
T(t-s)J[f(s,Y_1(s-))-f(s,Y_2(s-))]\rmd s\right\|^2&\leq&
\frac{K^2L}{\omega^2}\cdot\sup_{s\in
\mathbb{R}}\mathbf{E}\|Y_1(s)-Y_2(s)\|^2.
\end{eqnarray*}
The 2nd and the 6th term can be estimated by
\begin{eqnarray*}
&&\mathbf{E} \left\| \int_{-\infty}^{t} T(t-s)P[g(s,Y_{1}(s-))-g(s,Y_{2}(s-))]\rmd W(s)\right\|^2\\
&\leq&K^{2} \left(\int_{-\infty}^t \rme^{-2\omega(t-s)} \mathbf{E}
\|(g(s,Y_{1}(s-))-g(s,Y_{2}(s-)))Q^{1/2}\|^2_{L(U, \mathcal L^2(\mathbf P,\mathbb H))}\rmd s\right)\\
&\leq&\frac{K^{2}L}{2\omega}\cdot\sup_{s\in \mathbb{R}}
\mathbf{E}\left\| Y_{1}(s)-Y_{2}(s)\right\|^{2}
\end{eqnarray*}
and
\begin{eqnarray*}
&&\mathbf{E} \left\| \int_{t}^{+\infty} T(t-s)J[g(s,Y_{1}(s-))-g(s,Y_{2}(s-))]\rmd W(s)\right\|^2\\
&\leq&K^{2} \left(\int_{t}^{+\infty} \rme^{2\omega(t-s)} \mathbf{E}
\|(g(s,Y_{1}(s-))-g(s,Y_{2}(s-)))Q^{1/2}\|^2_{L(U, \mathcal L^2(\mathbf P,\mathbb H))}\rmd s\right)\\
&\leq&\frac{K^{2}L}{2\omega}\cdot\sup_{s\in \mathbb{R}}
\mathbf{E}\left\| Y_{1}(s)-Y_{2}(s)\right\|^{2}.
\end{eqnarray*}
For the 3rd term, we have
\begin{eqnarray*}
&&\mathbf{E}\left\|\int_{-\infty}^{t}\int_{|x|_U<1}T(t-s)P[F(s,Y_{1}(s-),x)-F(s,Y_{2}(s-),x)]\widetilde{N}(\rmd
s,\rmd x)\right\|^{2}
\\
&\leq& K^{2} \int_{-\infty}^{t}\int_{|x|_U<1} \rme^{-2\omega(t-s)}
\mathbf{E} \left\|F(s,Y_{1}(s-),x)-F(s,Y_{2}(s-),x)\right\|^{2} \nu
(\rmd x) \rmd s
\\
&\leq& K^{2}L \int_{-\infty}^{t} \rme^{-2\omega(t-s)}  \rmd s \cdot
\sup_{s \in \mathbb{R}} \mathbf{E} \left\| Y_{1}(s)-Y_{2}(s)
\right\|^{2}
\\
&\leq& \frac{K^{2}L}{2\omega} \cdot \sup_{s \in \mathbb{R}}
\mathbf{E} \left\| Y_{1}(s)-Y_{2}(s) \right\|^{2}.
\end{eqnarray*}
For the 4th term, we have
\begin{eqnarray*}
&&\mathbf{E} \left\| \int_{-\infty}^{t}\int_{|x|_U\geq
1}T(t-s)P[G(s,Y_{1}(s-),x)-G(s,Y_{2}(s-),x)]{N}(\rmd s,\rmd
x)\right\|^{2}
\\
&\leq& 2\mathbf{E} \left\| \int_{-\infty}^{t}\int_{|x|_U\geq
1}T(t-s)
P[G(s,Y_{1}(s-),x)-G(s,Y_{2}(s-),x)]{\widetilde{N}}(\rmd s,\rmd x) \right\|^{2}  \\
&&+2\mathbf{E} \left\| \int_{-\infty}^{t}\int_{|x|_U\geq
1}T(t-s)P[G(s,Y_{1}(s-),x)-G(s,Y_{2}(s-),x)]\nu (\rmd x) \rmd s
\right\|^{2}
\\
&\leq& \frac{K^{2}L}{\omega} \cdot \sup_{s \in \mathbb{R}}
\mathbf{E} \left\| Y_{1}(s)-Y_{2}(s) \right\|^{2}
\\
&&+2K^{2} \mathbf{E} \bigg( \int_{-\infty}^{t}\int_{|x|_U\geq
1}\rme^{\frac{-\omega(t-s)}{2}} \\
&&\qquad \cdot
\rme^{\frac{-\omega(t-s)}{2}}\|G(s,Y_{1}(s-),x)-G(s,Y_{2}(s-),x)\|\nu
(\rmd x) \rmd s \bigg)^{2}
\\
&\leq& \frac{K^{2}L}{\omega} \cdot \sup_{s \in \mathbb{R}}
\mathbf{E} \left\| Y_{1}(s)-Y_{2}(s) \right\|^{2} +
2K^{2}\int_{-\infty}^{t}\rme^{-\omega(t-s)}\rmd s \cdot b
\\
&&\cdot \int_{-\infty}^{t}\int_{|x|_U\geq 1}\rme^{-\omega(t-s)}
\mathbf{E} \left\|G(s,Y_{1}(s-),x)-G(s,Y_{2}(s-),x)\right\|^{2}\nu
(\rmd x) \rmd s
\\
&\leq& (\frac{K^{2}L}{\omega}+ \frac{2K^{2}Lb}{\omega^{2}}) \cdot
\sup_{s \in \mathbb{R}} \mathbf{E} \left\| Y_{1}(s)-Y_{2}(s)
\right\|^{2}.
\end{eqnarray*}

Similarly, the 7th and the 8th term can be estimated by
\begin{eqnarray*}
&&\mathbf{E}\left\|\int_{t}^{+\infty}\int_{|x|_U<1}T(t-s)J[F(s,Y_{1}(s-),x)-F(s,Y_{2}(s-),x)]\widetilde{N}(\rmd
s,\rmd x)\right\|^{2}
\\
&\leq& \frac{K^{2}L}{2\omega} \cdot \sup_{s \in \mathbb{R}}
\mathbf{E} \left\| Y_{1}(s)-Y_{2}(s) \right\|^{2}
\end{eqnarray*}
and
\begin{eqnarray*}
&&\mathbf{E} \left\| \int_{t}^{+\infty}\int_{|x|_U\geq
1}T(t-s)J[G(s,Y_{1}(s-),x)-G(s,Y_{2}(s-),x)]{N}(\rmd s,\rmd
x)\right\|^{2}
\\
&\leq& 2\mathbf{E} \left\| \int_{t}^{+\infty}\int_{|x|_U\geq
1}T(t-s)J[G(s,Y_{1}(s-),x)-G(s,Y_{2}(s-),x)]
{\widetilde{N}}(\rmd s,\rmd x) \right\|^{2}  \\
&&+2\mathbf{E} \left\| \int_{t}^{+\infty}\int_{|x|_U\geq
1}T(t-s)J[G(s,Y_{1}(s-),x)-G(s,Y_{2}(s-),x)]\nu (\rmd x) \rmd s
\right\|^{2}
\\
&\leq& (\frac{K^{2}L}{\omega}+ \frac{2K^{2}Lb}{\omega^{2}}) \cdot
\sup_{s \in \mathbb{R}} \mathbf{E} \left\| Y_{1}(s)-Y_{2}(s)
\right\|^{2}.
\end{eqnarray*}

So for any $t\in\R$,
\[
\mathbf{E}\|(\mathcal{S}Y_1)(t)-(\mathcal{S}Y_2)(t)\|^2 \leq \left[
\frac{16K^{2}L}{\omega^{2}} (1+2b) + \frac{32K^{2}L}{\omega}\right]
\sup_{s \in \mathbb{R}} \mathbf{E} \left\| Y_{1}(s)-Y_{2}(s)
\right\|^{2} ,
\]
that is,
\begin{equation}\label{equ1r}
\|(\mathcal{S}Y_1)(t)-(\mathcal{S}Y_2)(t)\|_2^2\le\eta\cdot\sup_{s\in
\mathbb{R}}\|Y_1(s)-Y_2(s)\|_2^2,
\end{equation}
where $\eta:=\frac{16K^{2}L}{\omega^{2}} (1+2b) +
\frac{32K^{2}L}{\omega}$. Since
\begin{equation}\label{equ2r}
\sup_{s\in \mathbb{R}}\|Y_1(s)-Y_2(s)\|_2^2\le (\sup_{s\in
\mathbb{R}}\|Y_1(s)-Y_2(s)\|_2)^2,
\end{equation}
it follows from  \eqref{equ1r} and \eqref{equ2r} that for arbitrary
$t\in\mathbb R$,
\[
\|\mathcal S(Y_1)(t)-\mathcal
S(Y_2)(t)\|_2\le\sqrt{\eta}~\|Y_1-Y_2\|_{\infty}.
\]
Thus
\[
\|\mathcal{S}Y_1-\mathcal{S}Y_2\|_{\infty}=\sup_{t\in
\mathbb{R}}\|\mathcal S(Y_1)(t)-\mathcal
S(Y_2)(t)\|_2\le\sqrt{\eta}~\|Y_1-Y_2\|_{\infty}.
\]
By \eqref{Bar} and the fact $\eta<1$, $\mathcal{S}$ is a contraction
mapping on $C_b(\mathbb R;\mathcal{L}^2(\mathbf{P},\mathbb{H}))$.
Therefore, there exists a unique $v\in C_b(\mathbb
R;\mathcal{L}^2(\mathbf{P},\mathbb{H}))$ with $\mathcal{S}v=v$,
which a the unique $\cal L^2$-bounded and $\cal L^2$-continuous
solution of \eqref{dxr}.
\end{proof}

\begin{thm}\label{thnon2}
Let (H1)-(H3) be satisfied. Then the unique $\cal L^2$-bounded mild
solution of \eqref{dxr} obtained in Theorem \ref{thnon1} is almost
periodic in distribution, provided
\begin{equation}\label{Banar}
\frac{1+2b}{\omega^{2}} + \frac{2}{\omega}<\frac{1}{32K^{2}L}.
\end{equation}
\end{thm}

\begin{proof}
For given sequence of real numbers $\{s'_n\}$, since  $f, g$ are
almost periodic and $F, G$ are Poisson almost periodic, there exist
a subsequence $\{s_n\}$ of $\{s'_n\}$ and and $\cal L^2$-continuous
stochastic process $\widetilde{f}$, $\widetilde{g}$ and continuous
functions $\widetilde{F}$, $\widetilde{G}$ in the sense of
Definition \ref{defap} (3), such that
\begin{equation}\label{uap-f}
\lim_{n\rightarrow \infty} \sup_{t\in\R,Y\in\cal K}
\mathbf{E}\|f(t+s_n,Y)-\widetilde{f}(t,Y)\|^2=0,
\end{equation}
\begin{equation}\label{uap-g}
\lim_{n\rightarrow \infty}\sup_{t\in\R,Y\in\cal
K}\mathbf{E}\|(g(t+s_{n},Y)-\widetilde{g}(t,Y))Q^{1/2}\|_{L(U,\mathcal{L}^{2}(\mathbf{P},
\mathbb{H}))}^2=0,
\end{equation}
\begin{align}\label{uap-F}
\lim_{n \rightarrow \infty}\sup_{t\in\R,Y\in\cal K}
\int_{|x|_U<1}\mathbf{E}
\|F(t+s_{n},Y,x)-\widetilde{F}(t,Y,x)\|^2\nu(\rmd x)&=0,
\end{align}
and
\begin{align}\label{uap-G}
\lim_{n\rightarrow \infty} \sup_{t\in\R,Y\in\cal K}
\int_{|x|_U\geq1}\mathbf{E}\|G(t+s_n,Y,x)-\widetilde{G}(t,Y,x)\|^2\nu(\rmd
x)&=0
\end{align}
hold for any bounded set $\cal K\subset\cal L^2(\mathbf P, \mathbb
H)$.

Let $\tilde Y(\cdot)$ be the solution of the following stochastic
integral equation\footnote{Note that $\widetilde{f}, \widetilde{g},
\widetilde{F}, \widetilde{G}$ are uniformly bounded in $t$ when
$Y=0$ and Lipschitz in $Y$ with the same Lipschitz constants as that
of $f,g,F,G$, so by Theorem \ref{thnon1} we know this integral
equation admits a unique $\cal L^2$-bounded solution. }
\begin{align*}
\tilde{Y}(t):=&\int_{-\infty}^tT(t-s)P\widetilde{f}(s,\tilde
Y(s-))\rmd
s+\int_{-\infty}^tT(t-s)P\widetilde{g}(s,\tilde Y(s-))\rmd W(s)\\
&+\int_{-\infty}^t\int_{|x|_U <
1}T(t-s)P\widetilde{F}(s,\tilde Y(s-),x)\widetilde{N}(\rmd s,\rmd x)\\
& +\int_{-\infty}^t\int_{|x|_U \geq 1}T(t-s)P\widetilde{G}(s,\tilde
Y(s-),x)N(\rmd
s,\rmd x)\\
& -\int^{+\infty}_tT(t-s)J\widetilde{f}(s,\tilde Y(s-))\rmd
s -\int^{+\infty}_tT(t-s)J\widetilde{g}(s,\tilde Y(s-))\rmd W(s)\\
&-\int^{+\infty}_t\int_{|x|_U <
1}T(t-s)J\widetilde{F}(s,\tilde Y(s-),x)\widetilde{N}(\rmd s,\rmd x)\\
& -\int^{+\infty}_t\int_{|x|_U \geq 1}T(t-s)J\widetilde{G}(s,\tilde
Y(s-),x)N(\rmd s,\rmd x)
\end{align*}

For each $\sigma\in\mathbb R$, denote
$W_n(\sigma):=W(\sigma+s_n)-W(s_n)$,
$N_n(\sigma,x):=N(\sigma+s_{n},x)-N(s_{n},x)$,
$\widetilde{N}_n(\sigma,x):=\widetilde{N}(\sigma+s_n,x)-\widetilde{N}(s_n,x)$.
Then $W_n$ is a $Q$-Wiener process with the same distribution as
$W$, and $N_{n}$ is a Poisson random measure with the same
distribution as $N$ and the corresponding compensated Poisson random
measure being $\widetilde{N}_n$. Let $\sigma=s-s_n$, then we have
\begin{align*}
{Y}(t+s_n)=&\int_{-\infty}^tT(t-\sigma)P{f}(\sigma+s_n,Y(\sigma+s_n-))\rmd
\sigma\\
&+\int_{-\infty}^tT(t-\sigma)P{g}(\sigma+s_n,Y(\sigma+s_n-))\rmd W_n(\sigma)\\
&+\int_{-\infty}^t\int_{|x|_U <
1}T(t-\sigma)P{F}(\sigma+s_n,Y(\sigma+s_n-),x)\widetilde{N}_n(\rmd \sigma,\rmd x)\\
&+\int_{-\infty}^t\int_{|x|_U \geq
1}T(t-\sigma)P{G}(\sigma,Y(\sigma+s_n-),x)N_n(\rmd
\sigma,\rmd x)\\
&-\int^{+\infty}_t T(t-\sigma)J{f}(\sigma+s_n,Y(\sigma+s_n-))\rmd
\sigma\\
&- \int^{+\infty}_t T(t-\sigma)J{g}(\sigma+s_n,Y(\sigma+s_n-))\rmd W_n(\sigma)\\
&-\int^{+\infty}_t\int_{|x|_U <
1}T(t-\sigma)J{F}(\sigma+s_n,Y(\sigma+s_n-),x)\widetilde{N}_n(\rmd \sigma,\rmd x)\\
&-\int^{+\infty}_t \int_{|x|_U \geq
1}T(t-\sigma)J{G}(\sigma,Y(\sigma+s_n-),x)N_n(\rmd \sigma,\rmd x).
\end{align*}
Consider the stochastic process $Y_n(\cdot)$ which satisfies the
following stochastic integral equation
\begin{align*}
{Y}_n(t)=&\int_{-\infty}^tT(t-\sigma)P{f}(\sigma+s_n,Y_n(\sigma-))\rmd
\sigma\\
&+\int_{-\infty}^tT(t-\sigma)P{g}(\sigma+s_n,Y_n(\sigma-))\rmd W(\sigma)\\
&+\int_{-\infty}^t\int_{|x|_U <
1}T(t-\sigma)P{F}(\sigma+s_n,Y_n(\sigma-),x)\widetilde{N}(\rmd \sigma,\rmd x)\\
&+\int_{-\infty}^t\int_{|x|_U \geq
1}T(t-\sigma)P{G}(\sigma,Y_n(\sigma-),x)N(\rmd
\sigma,\rmd x)\\
&-\int^{+\infty}_tT(t-\sigma)J{f}(\sigma+s_n,Y_n(\sigma-))\rmd
\sigma\\
&-\int^{+\infty}_tT(t-\sigma)J{g}(\sigma+s_n,Y_n(\sigma-))\rmd W(\sigma)\\
&-\int^{+\infty}_t\int_{|x|_U <
1}T(t-\sigma)J{F}(\sigma+s_n,Y_n(\sigma-),x)\widetilde{N}(\rmd \sigma,\rmd x)\\
&-\int^{+\infty}_t\int_{|x|_U \geq
1}T(t-\sigma)J{G}(\sigma,Y_n(\sigma-),x)N(\rmd \sigma,\rmd x).
\end{align*}
Note that $Y(t+s_n)$ and $Y_n(t)$ share the same distribution for
each $t\in\R$, the functions $f(\cdot+s_n,\cdot)$ and
$g(\cdot+s_n,\cdot)$ are uniformly square-mean almost periodic and
Lipschitz in $Y$ with the same Lipschitz constants as that of $f,g$,
and  $F(\cdot+s_n,\cdot,\cdot)$, $G(\cdot+s_n,\cdot,\cdot)$ are
uniformly Poisson square-mean almost periodic and satisfy the same
Lipschitz condition as that of $F,G$. So similar to Theorem
\ref{thnon1} this $Y_n(\cdot)$ exists, is unique and $\cal
L^2$-bounded.

It follows from It\^o's isometry and the properties of the integral
for Poisson random measures that
\begin{align*}
&\mathbf{E}\|Y_n(t)-\tilde{Y}(t)\|^2\\
\le&8\mathbf{E}\left\|
\int_{-\infty}^tT(t-\sigma)P[f(\sigma+s_n,Y_n(\sigma-))-\widetilde{f}(\sigma,\tilde Y(\sigma-))]\rmd\sigma\right\|^2\\
& + 8\mathbf{E}\left\|
\int^{+\infty}_tT(t-\sigma)J[f(\sigma+s_n,Y_n(\sigma-))-\widetilde{f}(\sigma,\tilde Y(\sigma-))]\rmd\sigma\right\|^2\\
&+8\mathbf{E}\left\|
\int_{-\infty}^tT(t-\sigma)P[g(\sigma+s_n,Y_n(\sigma-))-\widetilde{g}(\sigma,\tilde Y(\sigma-))]\rmd{W}(\sigma)\right\|^2\\
&+8\mathbf{E}\left\|
\int^{+\infty}_tT(t-\sigma)J[g(\sigma+s_n,Y_n(\sigma-))-\widetilde{g}(\sigma,\tilde Y(\sigma-))]\rmd{W}(\sigma)\right\|^2\\
&+8\mathbf{E} \left\|\int_{-\infty}^{t} \int_{|x|_U<1}T(t-\sigma)P
[F(\sigma +s_{n},Y_n(\sigma-),x)-\widetilde{F}(\sigma, \tilde Y(\sigma-),x)]\widetilde{N}(\rmd \sigma,\rmd x) \right\|^{2} \\
&+8\mathbf{E} \left\|\int^{+\infty}_{t} \int_{|x|_U<1}T(t-\sigma)J
[F(\sigma +s_{n},Y_n(\sigma-),x)-\widetilde{F}(\sigma, \tilde Y(\sigma-),x)]\widetilde{N}(\rmd \sigma,\rmd x) \right\|^{2} \\
&+8\mathbf{E} \left\|\int_{-\infty}^{t} \int_{|x|_U\geq
1}T(t-\sigma)P
[G(\sigma +s_{n},Y_n(\sigma-),x)-\widetilde{G}(\sigma, \tilde Y(\sigma-),x)]{N}(\rmd \sigma,\rmd x) \right\|^{2}\\
&+8\mathbf{E} \left\|\int^{+\infty}_{t} \int_{|x|_U\geq
1}T(t-\sigma)J
[G(\sigma +s_{n},Y_n(\sigma-),x)-\widetilde{G}(\sigma, \tilde Y(\sigma-),x)]{N}(\rmd \sigma,\rmd x) \right\|^{2}\\
=: & I_1 + I_2 + I_3 +I_4.
\end{align*}

By the Cauchy-Schwarz inequality we have the following estimate for
$I_1$:
\begin{align*}
I_1 \le  & 16\mathbf{E}\left\|
\int_{-\infty}^tT(t-\sigma)P[f(\sigma+s_n,Y_n(\sigma-))-f(\sigma+s_n,\tilde Y(\sigma-))]\rmd\sigma\right\|^2\\
 & + 16\mathbf{E}\left\|
\int_{-\infty}^tT(t-\sigma)P[f(\sigma+s_n,\tilde Y(\sigma-))-\widetilde{f}(\sigma,\tilde Y(\sigma-))]\rmd\sigma\right\|^2\\
& +16\mathbf{E}\left\|
\int^{+\infty}_tT(t-\sigma)J[f(\sigma+s_n,Y_n(\sigma-))-f(\sigma+s_n,\tilde Y(\sigma-))]\rmd\sigma\right\|^2\\
 & + 16\mathbf{E}\left\|
\int^{+\infty}_tT(t-\sigma)J[f(\sigma+s_n,\tilde Y(\sigma-))-\widetilde{f}(\sigma,\tilde Y(\sigma-))]\rmd\sigma\right\|^2\\
\le & 16
\mathbf{E}\left(\int_{-\infty}^t\|T(t-\sigma)P\|\cdot\|f(\sigma+s_n,Y_n(\sigma-))-f(\sigma+s_n,\tilde
Y(\sigma-))\|
\rmd\sigma\right)^2\\
&+ 16
\mathbf{E}\left(\int_{-\infty}^t\|T(t-\sigma)P\|\cdot\|f(\sigma+s_n,\tilde
Y(\sigma-))-\widetilde{f}(\sigma,\tilde Y(\sigma-))\|
\rmd\sigma\right)^2\\
& +16
\mathbf{E}\left(\int^{+\infty}_t\|T(t-\sigma)J\|\cdot\|f(\sigma+s_n,Y_n(\sigma-))-f(\sigma+s_n,\tilde
Y(\sigma-))\|
\rmd\sigma\right)^2\\
&+ 16
\mathbf{E}\left(\int^{+\infty}_t\|T(t-\sigma)J\|\cdot\|f(\sigma+s_n,\tilde
Y(\sigma-))-\widetilde{f}(\sigma,\tilde Y(\sigma-))\|
\rmd\sigma\right)^2\\
\le & 16 K^2 \int_{-\infty}^t \rme^{\frac{-\omega(t-\sigma)}{2}\cdot
2}
\rmd \sigma \\
& \quad\times \int_{-\infty}^t
\rme^{\frac{-\omega(t-\sigma)}{2}\cdot 2}\cdot
\mathbf{E}\|f(\sigma+s_n,Y_n(\sigma-))-f(\sigma+s_n,\tilde Y(\sigma-))\|^2 \rmd \sigma\\
&+ 16 K^2 \int_{-\infty}^t \rme^{\frac{-\omega(t-\sigma)}{2}\cdot 2}
\rmd \sigma \\
&\quad\times \int_{-\infty}^t \rme^{\frac{-\omega(t-\sigma)}{2}\cdot
2}\cdot
\mathbf{E}\|f(\sigma+s_n,\tilde Y(\sigma-))-\widetilde{f}(\sigma,\tilde Y(\sigma-))\|^2 \rmd \sigma\\
&+ 16 K^2 \int^{+\infty}_t \rme^{\frac{\omega(t-\sigma)}{2}\cdot 2}
\rmd \sigma \\
&\quad\times \int^{+\infty}_t \rme^{\frac{\omega(t-\sigma)}{2}\cdot
2}\cdot
\mathbf{E}\|f(\sigma+s_n,Y_n(\sigma-))-f(\sigma+s_n,\tilde Y(\sigma-))\|^2 \rmd \sigma\\
&+ 16 K^2 \int^{+\infty}_t \rme^{\frac{\omega(t-\sigma)}{2}\cdot 2}
\rmd \sigma \\
&\quad\times \int^{+\infty}_t \rme^{\frac{\omega(t-\sigma)}{2}\cdot
2}\cdot
\mathbf{E}\|f(\sigma+s_n,\tilde Y(\sigma-))-\widetilde{f}(\sigma,\tilde Y(\sigma-))\|^2 \rmd \sigma\\
\le & \frac{16K^2L}{\omega} \int_{-\infty}^t
\rme^{-\omega(t-\sigma)}\cdot
\mathbf{E}\|Y_n(\sigma-)-\tilde Y(\sigma-)\|^2 \rmd \sigma \\
& +\frac{16K^2L}{\omega} \int^{+\infty}_t
\rme^{\omega(t-\sigma)}\cdot
\mathbf{E}\|Y_n(\sigma-)-\tilde Y(\sigma-)\|^2 \rmd \sigma + \cal E_1^n(t)\\
\le & \frac{32K^2L}{\omega^2} \sup_{\sigma\in\R}
\mathbf{E}\|Y_n(\sigma)-\tilde Y(\sigma)\|^2 + \cal E_1^n (t),
\end{align*}
where
\begin{align*}
\cal E_1^n (t):=& \frac{16K^2}{\omega} \int_{-\infty}^t
\rme^{-\omega(t-\sigma)}\cdot\mathbf{E}\|f(\sigma+s_n,\tilde Y(\sigma-))-\widetilde{f}(\sigma,\tilde Y(\sigma-))\|^2 \rmd \sigma\\
& + \frac{16K^2}{\omega} \int^{+\infty}_t
\rme^{\omega(t-\sigma)}\cdot\mathbf{E}\|f(\sigma+s_n,\tilde
Y(\sigma-))-\widetilde{f}(\sigma,\tilde Y(\sigma-))\|^2 \rmd \sigma.
\end{align*}
By \eqref{uap-f}and the $\cal L^2$-boundedness of  $\tilde
Y(\cdot)$, we have $\lim_{n\to\infty}\sup_{t\in\R}\cal E_1^n(t) =0$.

For $I_2$, by It\^o's isometry we have:
\begin{align*}
I_2 \le & 16\mathbf{E}\left\|
\int_{-\infty}^tT(t-\sigma)P[g(\sigma+s_n,Y_n(\sigma-))-g(\sigma+s_n,\tilde Y(\sigma-))]\rmd W(\sigma)\right\|^2\\
 & + 16\mathbf{E}\left\|
\int_{-\infty}^tT(t-\sigma)P[g(\sigma+s_n,\tilde Y(\sigma-))-\widetilde{g}(\sigma,\tilde Y(\sigma-))]\rmd W(\sigma)\right\|^2\\
&+ 16\mathbf{E}\left\|
\int^{+\infty}_tT(t-\sigma)J[g(\sigma+s_n,Y_n(\sigma-))-g(\sigma+s_n,\tilde Y(\sigma-))]\rmd W(\sigma)\right\|^2\\
 & + 16\mathbf{E}\left\|
\int^{+\infty}_tT(t-\sigma)J[g(\sigma+s_n,\tilde
Y(\sigma-))-\widetilde{g}(\sigma,\tilde Y(\sigma-))]\rmd
W(\sigma)\right\|^2
\\
\le & 16  \mathbf{E}\int_{-\infty}^t
\|T(t-\sigma)P\|^2\cdot\|g(\sigma+s_n,Y_n(\sigma-))-g(\sigma+s_n,\tilde
Y(\sigma-))Q^{1/2}\|^2
\rmd\sigma\\
&+ 16 \mathbf{E}
\int_{-\infty}^t\|T(t-\sigma)P\|^2\cdot\|g(\sigma+s_n,\tilde
Y(\sigma-))-\widetilde{g}(\sigma,\tilde Y(\sigma-))Q^{1/2}\|^2
\rmd\sigma \\
&+16  \mathbf{E}\int^{+\infty}_t
\|T(t-\sigma)J\|^2\cdot\|g(\sigma+s_n,Y_n(\sigma-))-g(\sigma+s_n,\tilde
Y(\sigma-))Q^{1/2}\|^2
\rmd\sigma\\
&+ 16 \mathbf{E}
\int^{+\infty}_t\|T(t-\sigma)J\|^2\cdot\|g(\sigma+s_n,\tilde
Y(\sigma-))-\widetilde{g}(\sigma,\tilde Y(\sigma-))Q^{1/2}\|^2
\rmd\sigma \\
\le & 16 K^2L  \int_{-\infty}^t \rme^{-2\omega(t-\sigma)}
\cdot \mathbf{E}\|Y_n(\sigma-)-\tilde Y(\sigma-)\|^2 \rmd \sigma\\
& + 16 K^2L  \int^{+\infty}_t \rme^{2\omega(t-\sigma)} \cdot
\mathbf{E}\|Y_n(\sigma-)-\tilde Y(\sigma-)\|^2 \rmd \sigma
+ \cal E_2^n (t)\\
\le & \frac{16K^2L}{\omega} \sup_{\sigma\in\R}
\mathbf{E}\|Y_n(\sigma)-\tilde Y(\sigma)\|^2 +  \cal E_2^n(t)
\end{align*}
with
\begin{align*}
\cal E_2^n(t) := & 16 K^2 \int_{-\infty}^t \rme^{-2\omega(t-\sigma)}
\cdot \mathbf{E}\|g(\sigma+s_n,\tilde Y(\sigma-))-\widetilde{g}(\sigma,\tilde Y(\sigma-))Q^{1/2}\|^2 \rmd \sigma \\
& + 16 K^2 \int^{+\infty}_t \rme^{2\omega(t-\sigma)} \cdot
\mathbf{E}\|g(\sigma+s_n,\tilde
Y(\sigma-))-\widetilde{g}(\sigma,\tilde Y(\sigma-))Q^{1/2}\|^2 \rmd
\sigma.
\end{align*}
Similar to $\cal E_1^n$,  we have
\begin{align*}
\cal E_2^n(t) \to 0 \qquad \hbox{as } n\to \infty
\end{align*}
uniformly in $t\in\R$.

By the properties of the integral for Poisson random measures and
\eqref{F} we have
\begin{align*}
I_3 \le & 16 \mathbf{E}\bigg\|\int_{-\infty}^{t}\int_{|x|_U<1}T(t-\sigma)P[F(\sigma +s_{n},Y_n(\sigma-),x)\\
& \qquad -F(\sigma+s_n, \tilde Y(\sigma-),x)]\widetilde{N}(\rmd
\sigma,\rmd x)\bigg\|^{2}
\\
& + 16 \mathbf{E}\bigg\|\int_{-\infty}^{t}\int_{|x|_U<1}T(t-\sigma)P[F(\sigma +s_{n},\tilde Y(\sigma-),x)\\
&\qquad -\widetilde{F}(\sigma, \tilde
Y(\sigma-),x)]\widetilde{N}(\rmd \sigma,\rmd x)\bigg\|^{2}
\\
& + 16 \mathbf{E}\bigg\|\int^{+\infty}_{t}\int_{|x|_U<1}T(t-\sigma)J[F(\sigma +s_{n},Y_n(\sigma-),x)\\
& \qquad -F(\sigma+s_n, \tilde Y(\sigma-),x)]\widetilde{N}(\rmd
\sigma,\rmd x)\bigg\|^{2}
\\
& + 16 \mathbf{E}\bigg\|\int^{+\infty}_{t}\int_{|x|_U<1}T(t-\sigma)J[F(\sigma +s_{n},\tilde Y(\sigma-),x)\\
&\qquad -\widetilde{F}(\sigma, \tilde
Y(\sigma-),x)]\widetilde{N}(\rmd \sigma,\rmd x)\bigg\|^{2}
\\
\leq & 16 K^{2} \int_{-\infty}^{t}\int_{|x|_U<1}
\rme^{-2\omega(t-\sigma)}
\mathbf{E} \bigg\|F(\sigma +s_{n},Y_n(\sigma-),x)\\
& \qquad -F(\sigma+s_n, \tilde Y(\sigma-),x)\bigg\|^{2} \nu (\rmd x)
\rmd \sigma
\\
&+ 16 K^{2} \int_{-\infty}^{t}\int_{|x|_U<1}
\rme^{-2\omega(t-\sigma)}
\mathbf{E} \bigg\|F(\sigma +s_{n},\tilde Y(\sigma-),x)\\
&\qquad -\widetilde F(\sigma, \tilde Y(\sigma-),x)\bigg\|^{2} \nu
(\rmd x) \rmd \sigma\\
&+16 K^{2} \int^{+\infty}_{t}\int_{|x|_U<1} \rme^{2\omega(t-\sigma)}
\mathbf{E} \bigg\|F(\sigma +s_{n},Y_n(\sigma-),x)\\
& \qquad -F(\sigma+s_n, \tilde Y(\sigma-),x)\bigg\|^{2} \nu (\rmd x)
\rmd \sigma
\\
&+ 16 K^{2} \int^{+\infty}_{t}\int_{|x|_U<1}
\rme^{2\omega(t-\sigma)}
\mathbf{E} \bigg\|F(\sigma +s_{n},\tilde Y(\sigma-),x)\\
&\qquad -\widetilde{F}(\sigma, \tilde Y(\sigma-),x)\bigg\|^{2} \nu
(\rmd x) \rmd \sigma
\\
\leq & 16K^{2}L \int_{-\infty}^{t} \rme^{-2\omega(t-\sigma)}\cdot \mathbf{E} \| Y_{n}(\sigma-)-\tilde Y(\sigma-)\|^{2} \rmd \sigma\\
& + 16K^{2}L \int^{+\infty}_{t} \rme^{2\omega(t-\sigma)}\cdot
\mathbf{E} \| Y_{n}(\sigma-)-\tilde Y(\sigma-)\|^{2} \rmd \sigma
+\cal E_3^n (t) \\
\le & \frac{16K^2 L}{\omega} \sup_{\sigma\in\R} \mathbf{E} \|
Y_{n}(\sigma)-\tilde Y(\sigma)\|^{2} +\cal E_3^n(t),
\end{align*}
where
\begin{align*}\label{e3}
\cal E_3^n(t) :=& 16 K^{2} \int_{-\infty}^{t}\int_{|x|_U<1}
\rme^{-2\omega(t-\sigma)}
\mathbf{E} \|F(\sigma +s_{n},\tilde Y(\sigma-),x)\\
&\qquad\qquad -\widetilde{F}(\sigma, \tilde Y(\sigma-),x)\|^{2} \nu
(\rmd x) \rmd \sigma\\
&+  16 K^{2} \int^{+\infty}_{t}\int_{|x|_U<1}
\rme^{2\omega(t-\sigma)}
\mathbf{E} \|F(\sigma +s_{n},\tilde Y(\sigma-),x)\\
&\qquad\qquad -\widetilde{F}(\sigma, \tilde Y(\sigma-),x)\|^{2} \nu
(\rmd x) \rmd \sigma.
\end{align*}
By\eqref{uap-F} and the $\cal L^2$-boundedness of $\tilde Y(\cdot)$,
we have $\lim_{n\to\infty}\sup_{t\in\R} \cal E_3^n (t)=0$.

Now let us estimate the last term $I_4$. It follows from the
properties of the integral of Poisson random measures, \eqref{G} and
the Cauchy-Schwarz inequality that
\begin{align*}
I_4 \le & 16 \mathbf{E} \bigg\| \int_{-\infty}^{t}\int_{|x|_U\geq 1}T(t-\sigma)P[G(\sigma+s_n,Y_{n}(\sigma-),x)\\
&\qquad -G(\sigma+s_n,\tilde Y(\sigma-),x)]{N}(\rmd \sigma,\rmd
x)\bigg\|^{2}
\\
& + 16 \mathbf{E} \bigg\| \int_{-\infty}^{t}\int_{|x|_U\geq 1}T(t-\sigma)P[G(\sigma+s_n,\tilde Y(\sigma-),x)\\
& \qquad -\tilde G(\sigma,\tilde Y(\sigma-),x)]{N}(\rmd \sigma,\rmd x)\bigg\|^{2}\\
& +16 \mathbf{E} \bigg\| \int^{+\infty}_{t}\int_{|x|_U\geq 1}T(t-\sigma)J[G(\sigma+s_n,Y_{n}(\sigma-),x)\\
&\qquad -G(\sigma+s_n,\tilde Y(\sigma-),x)]{N}(\rmd \sigma,\rmd
x)\bigg\|^{2}
\\
& + 16 \mathbf{E} \bigg\| \int^{+\infty}_{t}\int_{|x|_U\geq 1}T(t-\sigma)J[G(\sigma+s_n,\tilde Y(\sigma-),x)\\
& \qquad -\tilde G(\sigma,\tilde Y(\sigma-),x)]{N}(\rmd \sigma,\rmd
x)\bigg\|^{2}
\\
\leq& 32\mathbf{E} \bigg\| \int_{-\infty}^{t}\int_{|x|_U\geq 1}T(t-\sigma)P[G(\sigma+s_n,Y_{n}(\sigma-),x)\\
&\qquad -G(\sigma+s_n,\tilde Y(\sigma-),x)]{\widetilde{N}}(\rmd \sigma,\rmd x) \bigg\|^{2}  \\
&+32\mathbf{E} \bigg\| \int_{-\infty}^{t}\int_{|x|_U\geq 1}T(t-\sigma)P[G(\sigma+s_n,Y_{n}(\sigma-),x)\\
&\qquad -G(\sigma+s_n,\tilde Y(\sigma-),x)]\nu (\rmd x) \rmd \sigma\bigg\|^{2}\\
&+32\mathbf{E} \bigg\| \int_{-\infty}^{t}\int_{|x|_U\geq 1}T(t-\sigma)P[G(\sigma+s_n,\tilde Y(\sigma-),x)\\
&\qquad -\tilde G(\sigma,\tilde Y(\sigma-),x)]{\widetilde{N}}(\rmd \sigma,\rmd x) \bigg\|^{2}  \\
&+32\mathbf{E} \bigg\| \int_{-\infty}^{t}\int_{|x|_U\geq 1}T(t-\sigma)P[G(\sigma+s_n,\tilde Y(\sigma-),x)\\
&\qquad -\tilde G(\sigma,\tilde Y(\sigma-),x)]\nu (\rmd x) \rmd \sigma\bigg\|^{2}\\
&+32\mathbf{E} \bigg\| \int^{+\infty}_{t}\int_{|x|_U\geq 1}T(t-\sigma)J[G(\sigma+s_n,Y_{n}(\sigma-),x)\\
&\qquad -G(\sigma+s_n,\tilde Y(\sigma-),x)]{\widetilde{N}}(\rmd \sigma,\rmd x) \bigg\|^{2}  \\
&+32\mathbf{E} \bigg\| \int^{+\infty}_{t}\int_{|x|_U\geq 1}T(t-\sigma)J[G(\sigma+s_n,Y_{n}(\sigma-),x)\\
&\qquad -G(\sigma+s_n,\tilde Y(\sigma-),x)]\nu (\rmd x) \rmd \sigma\bigg\|^{2}\\
&+32\mathbf{E} \bigg\| \int^{+\infty}_{t}\int_{|x|_U\geq 1}T(t-\sigma)J[G(\sigma+s_n,\tilde Y(\sigma-),x)\\
&\qquad -\tilde G(\sigma,\tilde Y(\sigma-),x)]{\widetilde{N}}(\rmd \sigma,\rmd x) \bigg\|^{2}  \\
&+32\mathbf{E} \bigg\| \int^{+\infty}_{t}\int_{|x|_U\geq 1}T(t-\sigma)J[G(\sigma+s_n,\tilde Y(\sigma-),x)\\
&\qquad -\tilde G(\sigma,\tilde Y(\sigma-),x)]\nu (\rmd x) \rmd
\sigma\bigg\|^{2}
\\
\le & 32K^2L \int_{-\infty}^{t}\rme^{-2\omega(t-\sigma)} \mathbf{E}\|Y_n(\sigma-)-\tilde Y(\sigma-)\|^2\rmd \sigma \\
& + 32K^2 b  \int_{-\infty}^{t}\rme^{-\omega(t-\sigma)} \rmd \sigma
\cdot \int_{-\infty}^{t}\int_{|x|_U\geq 1} \rme^{-\omega(t-\sigma)}
\mathbf{E}\| G(\sigma+s_n,Y_{n}(\sigma-),x)\\
& \quad\qquad -G(\sigma+s_n,\tilde Y(\sigma-),x)\|^2 \nu(\rmd x) \rmd \sigma\\
&+ 32K^2 \int_{-\infty}^{t}\int_{|x|_U\geq 1}
\rme^{-2\omega(t-\sigma)}
\mathbf{E}\| G(\sigma+s_n,\tilde Y(\sigma-),x)\\
&\quad\qquad -\tilde G(\sigma,\tilde Y(\sigma-),x)\|^2 \nu(\rmd x)
\rmd \sigma
\\
&+ 32K^2 b  \int_{-\infty}^{t}\rme^{-\omega(t-\sigma)} \rmd \sigma \\
& \quad \cdot \int_{-\infty}^{t}\int_{|x|_U\geq 1}
\rme^{-\omega(t-\sigma)} \mathbf{E}\| G(\sigma+s_n,\tilde
Y(\sigma-),x)-\tilde G(\sigma,\tilde Y(\sigma-),x)\|^2 \nu(\rmd x)
\rmd \sigma
\\
& + 32K^2L \int^{+\infty}_{t}\rme^{2\omega(t-\sigma)} \mathbf{E}\|Y_n(\sigma-)-\tilde Y(\sigma-)\|^2\rmd \sigma \\
& + 32K^2 b  \int^{+\infty}_{t}\rme^{\omega(t-\sigma)} \rmd \sigma
\cdot \int^{+\infty}_{t}\int_{|x|_U\geq 1} \rme^{\omega(t-\sigma)}
\mathbf{E}\| G(\sigma+s_n,Y_{n}(\sigma-),x)\\
& \quad\qquad -G(\sigma+s_n,\tilde Y(\sigma-),x)\|^2 \nu(\rmd x) \rmd \sigma\\
&+ 32K^2 \int^{+\infty}_{t}\int_{|x|_U\geq 1}
\rme^{2\omega(t-\sigma)}
\mathbf{E}\| G(\sigma+s_n,\tilde Y(\sigma-),x)\\
&\quad\qquad -\tilde G(\sigma,\tilde Y(\sigma-),x)\|^2 \nu(\rmd x)
\rmd \sigma
\\
&+ 32K^2 b  \int^{+\infty}_{t}\rme^{\omega(t-\sigma)} \rmd \sigma \\
& \quad \cdot \int^{+\infty}_{t}\int_{|x|_U\geq 1}
\rme^{\omega(t-\sigma)} \mathbf{E}\| G(\sigma+s_n,\tilde
Y(\sigma-),x)-\tilde G(\sigma,\tilde Y(\sigma-),x)\|^2 \nu(\rmd x)
\rmd \sigma
\\
\le & 32K^2L \int_{-\infty}^{t} \rme^{-2\omega(t-\sigma)}\cdot\mathbf{E}\|Y_n(\sigma-)-\tilde Y(\sigma-)\|^2\rmd \sigma \\
& + \frac{32K^2bL}{\omega}\int_{-\infty}^{t}
\rme^{-\omega(t-\sigma)}\cdot
\mathbf{E}\|Y_n(\sigma-)-\tilde Y(\sigma-)\|^2\rmd \sigma\\
& +32K^2L \int^{+\infty}_{t} \rme^{2\omega(t-\sigma)}\cdot\mathbf{E}\|Y_n(\sigma-)-\tilde Y(\sigma-)\|^2\rmd \sigma \\
& + \frac{32K^2bL}{\omega}\int^{+\infty}_{t}
\rme^{\omega(t-\sigma)}\cdot
\mathbf{E}\|Y_n(\sigma-)-\tilde Y(\sigma-)\|^2\rmd \sigma + \cal E_4^n(t)\\
\le & \left( \frac{32 K^2 L}{\omega} +  \frac{64
K^2bL}{\omega^2}\right) \sup_{\sigma\in\R}
\mathbf{E}\|Y_n(\sigma)-\tilde Y(\sigma)\|^2\rmd \sigma + \cal
E_4^n(t),
\end{align*}
where
\begin{align*}
&\cal E_4^n (t):=   \\
& 32K^2 \int_{-\infty}^{t}\int_{|x|_U\geq 1}
\rme^{-2\omega(t-\sigma)} \mathbf{E}\| G(\sigma+s_n,\tilde
Y(\sigma-),x)-\tilde G(\sigma,\tilde Y(\sigma-),x)\|^2 \nu(\rmd x)
\rmd \sigma
\\
&+ \frac{32 K^2 b}{\omega}  \cdot \int_{-\infty}^{t}\int_{|x|_U\geq
1} \rme^{-\omega(t-\sigma)}
\mathbf{E}\| G(\sigma+s_n,\tilde Y(\sigma-),x)-\tilde G(\sigma,\tilde Y(\sigma-),x)\|^2 \nu(\rmd x) \rmd \sigma\\
&+32K^2 \int^{+\infty}_{t}\int_{|x|_U\geq 1}
\rme^{2\omega(t-\sigma)} \mathbf{E}\| G(\sigma+s_n,\tilde
Y(\sigma-),x)-\tilde G(\sigma,\tilde Y(\sigma-),x)\|^2 \nu(\rmd x)
\rmd \sigma
\\
&+ \frac{32 K^2 b}{\omega}  \cdot \int^{+\infty}_{t}\int_{|x|_U\geq
1} \rme^{\omega(t-\sigma)} \mathbf{E}\| G(\sigma+s_n,\tilde
Y(\sigma-),x)-\tilde G(\sigma,\tilde Y(\sigma-),x)\|^2 \nu(\rmd x)
\rmd \sigma.
\end{align*}
It follows from \eqref{uap-G} and the $\cal L^2$-boundedness of
$\tilde Y(\cdot)$ that  $\cal E_4^n(t) \to 0$ uniformly in $t\in\R$
as $n\to\infty$.

By the above estimates of $I_1$--$I_4$, we have
\begin{align*}
\mathbf{E}\|Y_n(t)-\tilde Y(t)\|^2 \le  \cal E^n(t) +
\left(\frac{32K^2L}{\omega^2} +\frac{64K^2L}{\omega} + \frac{64K^2 b
L}{\omega^2}\right)\cdot \sup_{\sigma\in\R}
\mathbf{E}\|Y_n(\sigma)-\tilde Y(\sigma)\|^2
\end{align*}
with $\cal E^n(t)=\sum_{i=1}^4 \cal E_i^n (t)$. So
\begin{align*}
& \sup_{t\in\R}\mathbf{E}\|Y_n(t)-\tilde Y(t)\|^2 \le   \sup_{t\in\R} \cal E^n(t) \\
&\qquad + \left(\frac{32K^2L}{\omega^2} +\frac{64K^2L}{\omega} +
\frac{64K^2 b L}{\omega^2}\right)\cdot \sup_{\sigma\in\R}
\mathbf{E}\|Y_n(\sigma)-\tilde Y(\sigma)\|^2.
\end{align*}
It follows from \eqref{Banar} and the fact
$\lim_{n\to\infty}\sup_{t\in\R}\cal E^n(t)=0$ that
\[
\sup_{t\in\R}\mathbf{E}\|Y_n(t)-\tilde Y(t)\|^2\to 0 \qquad \hbox{as
} n\to\infty.
\]
由于对任意的$t\in\R$, Since $Y(t+s_n)$ and $Y_n(t)$ share the same
distribution for each $t\in\R$, this yields that  故 $Y(t+s_n) \to
\tilde Y(t)$ in distribution uniformly in $t\in\R$ as $n\to\infty$.
The proof is complete.
\end{proof}

\section{Applications}\label{app}
\setcounter{equation}{0}

In this section, we give two examples to illustrate our results
 in this paper.

\begin{exam}\label{ex1}
Let us consider an ordinary differential equation perturbed
by a two-sided L\'evy noise:
\begin{align}\label{ex1}
\rmd y =   (Ay + f(t,y)) \rmd t + g(t,y)\rmd W + \int_{|x|<1} F(t,y,x) \widetilde N(\rmd t,\rmd x) + \int_{|x|\ge 1} G(t,y,x) N(\rmd t,\rmd x),\nonumber
\end{align}
where
\begin{align*}
{A} :=&
\left( \begin{array}{cc}
8&0 \\
0& -6
\end{array} \right), \quad
 \\
f(t, y):=& \left( \begin{array}{cc}
0\\
\frac{\cos\sqrt{2}t+ \sin\sqrt{3}t}{17+\cos\sqrt{5}t}\cdot\frac{y}{y^2+1}
\end{array} \right), \quad
g(t,y):= \left( \begin{array}{cc}
0\\
\frac{1}{12} \sin(y+\cos\sqrt{3}t+\cos\sqrt{2}t)
\end{array} \right), \\
F(t, y, x):=& \left( \begin{array}{cc}
0\\
\frac{1}{10} y
\end{array} \right), \quad
G(t,y,x):= \left( \begin{array}{cc}
0\\
\frac{1}{9}y\cdot\frac{\sin^2(\sqrt{3}t)}{3+\cos\sqrt{2}t+\cos\sqrt{5}t}
\end{array} \right), \\
\end{align*}
where $W$ is a one-dimensional Brownian motion and $N$ is a Poisson
random measure in $\R$ independent of $W$. It is clear that (H1) and
(H2) hold, that is, the semigroup by $A$ satisfies the exponential
dichotomy property with $K=1$ and $\omega=6$;  $f,g$ are almost
periodic in $t$ and $F,G$ are Poisson square-mean almost periodic in
$t$. The Lipschitz constants of $f,g,F,G$ in $y$ can be chosen as
$1/8, 1/12, 1/10, 1/9$, respectively, so the Lipschitz conditions in
(H3) are satisfied with $L=(1/8)^2=1/64$ provided
\[
\int_{|x|<1} \frac{1}{100} \nu(\rmd x) \le \frac{1}{64} \quad \hbox{and} \quad
\int_{|x|\ge 1} \frac{1}{81} \nu(\rmd x) \le \frac{1}{64},
\]
i.e.
\begin{equation*}\label{ex1paa}
\nu((-1,1)) < \frac{25}{16} \quad \hbox{and}\quad b\le \frac{81}{64}.
\end{equation*}
Then, the inequalities \eqref{Bar} and \eqref{Banar} become
\[
\frac{1+2b}{36} + \frac{2}{6} < 4 \quad \hbox{and} \quad
\frac{1+2b}{36} + \frac{2}{6} < 2,
\]
i.e. $b<59/2$. By Theorem \ref{thnon1}, \eqref{ex1} admits a unique
$\cal L^2$-bounded mild solution if $\nu((-1,1)) < 25/16$, $b\le
{81}/{64}$; furthermore, by Theorem \ref{thnon2}, this unique  $\cal
L^2$-bounded solution is  almost periodic in distribution under the
same conditions.
\end{exam}

\begin{exam}\label{ex2}
Let ${\Omega}\subset\R^n$, be a bounded subset whose boundary $\partial\Omega$
is of class $C^2$ and being locally on one side of $\Omega$.

Consider the parabolic stochastic partial
differential equation:
\begin{align*}
{\ud {u}}(t,\xi) =&\{A(\xi)u(t,\xi)+f(t,u(t,\xi))\}\ud t
+g(t,u(t,\xi))\ud W(t,\xi) \notag\\
&\qquad+h(t,u(t,\xi),Z(t,\xi))\ud Z(t,\xi)
\\
&\sum_{i,j=1}^n n_i(\xi)a_{ij}(t,\xi)D_i{ u}(t,\xi)=0,
\qquad t\in\R, \qquad\xi\in\partial\Omega,
\end{align*}
Here $D_i:=\frac{\rmd}{\rmd \xi_i}$,  $f,g,h$ are continuous
functions with additional properties which would be specified below,
$W$ is a $Q$-Wiener process on $L^2(\Omega)$ with ${\rm Tr}Q
<\infty$, and $Z$ is a L\'{e}vy pure jump process on $L^2(\Omega)$
which is independent of $W$, $n(\xi)=(n_1(\xi),n_2(\xi),\ldots,
n_n(\xi))$ is the outer unit normal vector, the family of operators
$A(\xi)$ are formally given by
\begin{align*}
A(\xi)=\sum_{i,j=1}^n \frac{\partial}{\partial
u_i}\left(a_{ij}(\xi)\frac{\partial}{\partial u_j}\right)+a_0(\xi),
\qquad t\in\R, \qquad\xi\in\Omega,
\end{align*}
and $a_0$, $a_{ij}(i,j=1,2,\ldots,n)$ satisfy the following conditions:

 \begin{itemize}
   \item [(i)] The coefficients $(a_{ij})_{i,j=1,\ldots,n}$ are symmetric, that is, $a_{ij}=a_{ji}$
               for all $i,j=1,\ldots,n$. Moreover, $a_{ij}\in \L^2(\P,C(\overline{\Omega}))$ $\bigcap$
               $\L^2(\P,C^1(\overline{\Omega}))$$\bigcap$$\L^2(\P,\L^2({\Omega}))$
               for all $i,j=1,\ldots,n$, and
               $a_0\in \L^2(\P,\L^2({\Omega}))$ $\bigcap$
               $\L^2(\P,C(\overline{\Omega}))$$\bigcap$$\L^2(\P,\L^1({\Omega}))$.

   \item [(ii)] There exists $\epsilon_0>0$ such that
                 \begin{align*}
                   \sum_{i,j=1}^n a_{ij}(\xi)\eta_i\eta_j\ge \epsilon_0|\eta|^2,
                 \end{align*}
                for all $\xi\in\overline{\Omega}$ and $\eta\in\R^n$.
 \end{itemize}
Denote $ \H= \L^2(\Omega)$. For each $t\in\R$ define an operator $A$
on $\L^2(\P,\H)$ by
\begin{align*}
\mathcal{D}(A)=\{u\in W^{2,2}(\Omega):\sum_{i,j=1}^n
n_i(\cdot)a_{ij}(\cdot)D_i  u=0
 \;\;\; \textrm{on} \;\;\; \partial\Omega\}
\end{align*}
and $AY=A(\xi)u(\xi)$ for all $Y\in \mathcal{D}(A)$.

Under above assumptions, the existence of a semigroup $T(t)$
satisfying (H1) is obtained, see \cite{MS03} for details.

Then the parabolic stochastic partial
differential equation can be written as an
abstract evolution equation
\begin{align}\label{aee}
\rmd Y &=(AY+ {F}(t,Y))\rmd t +{G}(t, Y)\rmd W + \int_{|z|_{L^2} <1}
\mathcal H (t,Y,z) \widetilde N(\rmd t, \rmd
z)\\
&\qquad + \int_{|z|_{L^2} \ge 1} \mathcal H (t,Y,z) N(\rmd t, \rmd
z)\nonumber
\end{align}
on the Hilbert space $\mathbb H$, where
\begin{align*}
{F}(t,Y):=& f(t, u), \quad
{G}(t,Y)\rmd W:= g(t,u)\rmd W\\
h(t,u,Z) \rmd Z
 :=&   \int_{|z|_{L^2} <1} \mathcal H (t,Y,z) \widetilde N(\rmd t, \rmd
z)+ \int_{|z|_{L^2} \ge 1}\mathcal H (t,Y,z) N(\rmd t, \rmd z)
\end{align*}
with
\begin{align*}
Z(t,\xi) &= \int_{|z|_{L^2} <1} z \widetilde N(t, \rmd z)+
\int_{|z|_{L^2} \ge 1} z N(t, \rmd z), \quad
\mathcal H (t,Y,z) = h(t,u,z) z.
\end{align*}
Here we assume for simplicity that the L\'evy pure jump process
$Z$ is decomposed as above by the L\'evy-It\^o
decomposition theorem.

We assume that the continuous functions $f(t,u)$ and $g(t,u)$  are
Lipschitz with respect to $u$  and uniformly almost periodic, then
the functions $\mathbb F$ and $\mathbb G$  in \eqref{aee} are
Lipshictz with respect to $Y$ and uniformly almost periodic. Let the
continuous function $h(t,u)$ be Lipschitz with respect to $u$ and
uniformly Poisson almost periodic, and the intensity measure $\nu$
of the Poisson process $Z$ be such that $\cal H$ is Lipschitz in $Y$
and uniformly Poisson square-mean almost periodic in the sense of
\eqref{F} and \eqref{G}. In particular, when $\nu$ is a finite
measure,  the required Lipschitz condition for  $\cal H$ holds. That
is, (H3) holds.

Since (H1)-(H3) are satisfied, the parabolic stochastic partial
differential equation  has a unique $\cal L^2$-bounded mild solution
$u\in\L^2(\P,\H)$, when  $K$ and the Lipschitz constants for $f, g,
h$ are  appropriately small; furthermore, this unique $\cal
L^2$-bounded solution is almost periodic in distribution, when the
Lipschitz constants for $f, g, h$ are smaller.

\end{exam}


\end{document}